\documentclass[12pt]{article}
\usepackage[utf8]{inputenc}

\title{Higher Specht bases for generalizations of the coinvariant ring}
\author{Maria Gillespie   \\
Colorado State University \\ maria.gillespie@colostate.edu
\and Brendon Rhoades\thanks{Supported by NSF grant DMS-1500838.} \\
UC San Diego \\
bprhoades@ucsd.edu
}
\date{\today}

\usepackage[margin=1in]{geometry}
\usepackage{graphicx}
\usepackage{amsmath,amssymb,amsthm,amsfonts}
\usepackage{comment}
\usepackage{youngtab}
\usepackage{ytableau}
\usepackage{young}
\usepackage{todonotes}

\newtheorem{theorem}{Theorem}[section]
\newtheorem{proposition}[theorem]{Proposition}
\newtheorem{corollary}[theorem]{Corollary}
\newtheorem{lemma}[theorem]{Lemma}
\newtheorem{conjecture}[theorem]{Conjecture}

\theoremstyle{definition}
\newtheorem{definition}[theorem]{Definition}
\newtheorem{example}[theorem]{Example}

\newcommand{\maj}{{\mathrm {maj}}}

\newcommand{\grFrob}{{\mathrm {grFrob}}}
\newcommand{\des}{{\mathrm {des}}}

\newcommand{\rev}{{\mathrm {rev}}}

\newcommand{\SYT}{{\mathrm {SYT}}}

\newcommand{\B}{\mathcal{B}}
\newcommand{\C}{\mathcal{C}}

\newcommand{\shape}{{\mathrm {shape}}}

\newcommand{\symm}{S}

\newcommand{\CC}{{\mathbb {C}}}
\newcommand{\QQ}{{\mathbb {Q}}}

\newcommand{\CCC}{{\mathcal{C}}}

\newcommand{\BBB}{{\mathcal{B}}}

\newcommand{\xx}{{\mathbf {x}}}

\newcommand{\defn}[1]{\textbf{#1}}

\newcommand{\cc}{\mathrm{cc}}

\newcommand{\cw}{\mathrm{cw}}
\newcommand{\sgn}{\mathrm{sgn}}
\newcommand{\SSYT}{\mathrm{SSYT}}
\newcommand{\Tab}{\mathrm{Tab}}

\begin{document}

\maketitle

\begin{abstract}
The classical coinvariant ring $R_n$ is defined as the quotient of a polynomial ring in $n$ variables by the positive-degree $S_n$-invariants.  It has a known basis that respects the decomposition of $R_n$ into irreducible $S_n$-modules, consisting of the \textit{higher specht polynomials} due to Ariki, Terasoma, and Yamada \cite{ATY}.  

We provide an extension of the higher Specht basis to the generalized coinvariant rings $R_{n,k}$ introduced in \cite{HRS}.  We also give a conjectured higher Specht basis for the Garsia-Procesi modules $R_\mu$, and we provide a proof of the conjecture in the case of two-row partition shapes $\mu$.  We then combine these results to give a 
 higher Specht basis for an infinite subfamily of the modules $R_{n,k,\mu}$ 
recently defined by Griffin \cite{Griffin}, which are a common generalization of $R_{n,k}$ and $R_{\mu}$.  
\end{abstract}

\section{Introduction and Background}
\label{sec:intro}

The \defn{Specht polynomials} provide one of the many ways of directly constructing the irreducible representations of the symmetric group $S_n$.  To define them, recall that a \textit{standard Young tableau} on a partition $\lambda$ of $n$ is a filling of the Young diagram of $\lambda$ with the numbers $1,\ldots,n$ that is increasing across rows and up columns (using the `French' convention for tableaux; see Figure \ref{fig:SYT}).  Given a standard Young tableau $T$, the Specht polynomial $F_T$ is defined as 
$$
F_T = \prod_C \prod_{\substack{i,j \in C,\\i<j}} (x_j - x_i),
$$
where the outer product is over all columns of $T$. For example, if $T$ is the tableau in Figure \ref{fig:SYT}, then $F_T = (x_1 - x_2)(x_1 - x_5)(x_2 - x_5)(x_3 - x_4)(x_6 - x_7)$.  

\begin{figure}
    \centering
    $\young(5,247,136)$
    \caption{A standard Young tableau $T$ of partition shape $\lambda=(3,3,1)$.}
    \label{fig:SYT}
\end{figure}

Given a fixed partition $\lambda$ of $n$, the set of Specht polynomials $\{F_T:T\text{ has shape }\lambda\}$  spans a subspace of $\mathbb{Q}[x_1,\ldots,x_n]$ isomorphic to the irreducible representation $V_\lambda$ of $S_n$ (under the usual $S_n$-action on the variables $x_i$).   Moreover, the polynomials $F_T$ are linearly independent, forming a basis of this representation.  (See \cite{Sagan} for proofs of these facts along with a general overview of symmetric group representation theory and symmetric function theory.)

\subsection{Higher Specht polynomials for the coinvariant ring}

The Specht polynomial construction has been generalized in \cite{ATY} to higher degree copies of $V_\lambda$ appearing in polynomial rings.  In particular, the $S_n$-module structure of the full polynomial ring is easily determined from that of the \defn{coinvariant ring} $$R_n=\mathbb{Q}[x_1,\ldots,x_n]/(e_1,\ldots,e_n).$$ Here $e_1,\ldots,e_n$ are the elementary symmetric functions in $x_1,\ldots,x_n$, defined by $$e_d=e_d(x_1,\ldots,x_n)=\sum_{1\le i_1<\cdots<i_d\le n} x_{i_1}\cdots x_{i_d}.$$

It is known that $R_n$, as an ungraded $S_n$-module, is isomorphic to the regular representation.  Thus each irreducible $S_n$-module $V_\lambda$ appears $\dim V_\lambda$ times, which is precisely the number of standard Young tableaux of shape $\lambda$.  Hence a basis of generalized Specht polynomials for $R_n$ should be indexed by pairs of standard Young tableaux of the same shape.  

To this end, in \cite{ATY} (and more succinctly described in \cite{ATYsmall}), Ariki, Terasoma, and Yamada defined the \defn{higher Specht polynomials} using the well-known \textit{cocharge}\footnote{In \cite{ATY}, the terminology used is `charge', but we use `cocharge' to be consistent with the original notation of Lascoux and Schutzenberger \cite{LascouxSchutzenberger}.} statistic.  We first recall the definition of cocharge for permutations and tableaux here.

\begin{definition}
 Let $\pi = \pi_1 \dots \pi_n$ be a permutation in $\symm_n$.  The \defn{cocharge word} $\cw(\pi) = c_1 \dots c_n$ is defined
 as follows.  Label the $1$ in $\pi$ with the subscript $0$.  Assuming the letter $i$ in $\pi$ has been labeled $j$, assign the letter $i+1$ 
 in $\pi$ the label $j$ if $\pi^{-1}_i < \pi^{-1}_{i+1}$ and $j+1$ if $\pi^{-1}_i > \pi^{-1}_{i+1}$.
Then $\cw(\pi) = c_1 \dots c_n$ is the list of labels, read left-to-right.
\end{definition}

\begin{definition}
If $S$ is a standard tableau, then $\cw(S)$ is the cocharge word of the \textit{reading word} of $S$, formed by concatenating the rows from top to bottom.  
\end{definition}
For example, if $S$ is the tableau at left in Figure \ref{fig:cocharge}, the reading word is $7346125$ so that the cocharge labeling is
$$7_3 \, 3_1 \, 4_1 \, 6_2 \, 1_0 \, 2_0 \, 5_1 $$ and $\cw(S) = 3112001$.  We can also represent $\cw(S)$ as a tableau by replacing the entry $i$ in $S$ with its cocharge label, as shown at right in Figure \ref{fig:cocharge}.
\begin{definition}
For any word $w$ or standard tableau $S$, we define its \defn{cocharge}, written $\cc(w)$ or $\cc(S)$ respectively, to be the sum of the labels in the cocharge word.
\end{definition}

\begin{figure}[t]
    \centering
    $\young(7,346,125)$\hspace{2cm}$\young(3,112,001)$
    \caption{A standard Young tableau $S$ at left, with its cocharge labels shown at right.}
    \label{fig:cocharge}
\end{figure}

Now suppose we have two standard tableaux $S$ and $T$ with the same shape.  Define the monomial $$\xx_T^{\cw(S)}=\prod_{i=1}^n x_i^{\cw(i)}$$ where $\cw(i)$ is the cocharge label in $\cw(S)$ in the same square as $i$ in $T$.  If $T$ is the tableau in Figure \ref{fig:SYT} and $S$ is at left in Figure \ref{fig:cocharge}, then $$\xx_T^{\cw(S)} = x_1^0 x_2^1 x_3^0 x_4^1 x_5^3 x_6^1 x_7^2=x_2x_4x_5^3x_6x_7^2.$$
Finally, define the {\em higher Specht polynomial} $F^S_T$ to be
\begin{equation}
F^S_T := \varepsilon_T \cdot \xx_T^{\cw(S)}
\end{equation}
where $\varepsilon_T \in \QQ[S_n]$ is the \defn{Young idempotent} corresponding to $T$.  That is, $$\varepsilon_T=\sum_{\tau \in C(T)}\sum_{\sigma \in R(T)}\sgn(\tau) \tau \sigma$$ where $C(T)\subseteq S_n$ is the group of \textit{column permutations}  of $T$ (those that send every number to another number in its column in $T$), and $R(T)\subseteq S_n$ is the group generated by row permutations.

\begin{example}
  Suppose $S$ is an SYT of shape $\lambda$ with the property that the numbers $1,\ldots,\lambda_1$ are in the bottom row, the numbers $\lambda_1+1,\lambda_1+2,\ldots,\lambda_1+\lambda_2$ are in the second, and so on.  Then its cocharge indices are $i-1$ in the $i$-th row for all $i$.  In this case, if $T$ is any SYT of shape $\lambda$, then we have $F_{T}^S=F_{T}$ where $F_T$ is the ordinary Specht polynomial defined above. 
\end{example}

If $V$ is a finite-dimensional $S_n$-module, there are unique multiplicities
$c_{\lambda}$ such that $V \cong \bigoplus_{\lambda \vdash n} c_{\lambda} V_{\lambda}. $
The {\em Frobenius character} of $V$ is the symmetric function
$\mathrm{Frob}(V) := \sum_{\lambda} c_{\lambda} s_{\lambda}$ obtained by replacing
each copy of $V_{\lambda}$ with the corresponding Schur function $s_{\lambda}$.
More generally, if $V = \bigoplus_{d \geq 0} V_d$ is a graded $S_n$-module with each
piece $V_d$ finite-dimensional, the {\em graded Frobenius character} of $V$
is $\grFrob(V;q) := \sum_{d \geq 0} \mathrm{Frob}(V) \cdot q^d$.

Let $\SYT(n)$ be the set of all standard Young tableaux with $n$ boxes.
In \cite{ATY}, Ariki, Terasoma, and Yamada proved that the set 
$$
\BBB_n := \{ F_T^S \,:\, S, T \in \SYT(n) \text{ have the same shape} \}
$$
descends to a basis for the classical coinvariant algebra $R_n$. Since $F_T^S$ is obtained by the action of the idempotent $\varepsilon_T$, it follows that the subspace generated by those elements $F_T^S$ with a fixed $T$ is a copy of the irreducible representation $V_\lambda$ where $\lambda=\shape(T)=\shape(S)$ is the partition shape of $S$ and $T$.  (See \cite[page 46]{FultonHarris}.)
As an immediate corollary, one obtains the known fact that the graded Frobenius character of $R_n$ is given by
$$
\grFrob(R_n; q) = \sum_{S \in \SYT(n)} q^{\cc(S)} s_{\mathrm{shape}(S)} 
= \sum_{S \in \SYT(n)} q^{\maj(S)} s_{\mathrm{shape}(S)}.
$$
Here $\maj$ is the major index (see Definition \ref{def:majdes} below).  The second equality follows from the equidistribution of cocharge
and major index on standard tableaux of a given shape (see \cite{Killpatrick}).

Our goal is to extend this setup to several important 
generalizations of the coinvariant ring.  To be precise, we define a \textit{higher Specht basis} of an arbitrary $S_n$-module as follows.

\begin{definition}
  Let $R$ be an $S_n$-module with decomposition $$R=\bigoplus_\lambda c_\lambda V_\lambda$$ into irreducible $S_n$-modules.  Then a \defn{higher Specht basis} of $R$ is a set of elements $\mathcal{B}$ such that there exists a decomposition $\mathcal{B}=\bigcup_\lambda \bigcup_{i=1}^{c_\lambda} \mathcal{B}_{\lambda,i}$ such that the elements of $\mathcal{B}_{\lambda_i}$ are a basis of the $i$-th copy of $V_\lambda$ in the decomposition of $R$.
\end{definition}

We now describe three important generalizations of the coinvariant ring in the following subsections, with the goal of constructing a higher Specht basis for each.

\subsection{The rings $R_{n,k}$}

For positive integers $k \leq n$, Haglund, Rhoades, and Shimozono \cite{HRS}
defined a quotient ring 
\begin{equation}
    R_{n,k} := \QQ[x_1, \dots, x_n]/I_{n,k}
\end{equation}
where $I_{n,k} \subseteq \QQ[x_1, \dots, x_n]$ is the ideal
\begin{equation}
    I_{n,k} := 
    \langle x_1^k, x_2^k, \dots, x_n^k, e_n, e_{n-1}, \dots, e_{n-k+1} \rangle.
\end{equation}
Since the ideal $I_{n,k}$ is homogeneous and $S_n$-stable, the ring $R_{n,k}$
is a graded $S_n$-module. When $k = n$, we recover the classical coinvariant ring, i.e.
$R_{n,n} = R_n$.
As an ungraded $S_n$-module, the ring $R_{n,k}$ is isomorphic \cite{HRS} 
to the permutation
action of $S_n$ on $k$-block ordered set partitions of $\{1, 2, \dots, n\}$.

The {\em Delta Conjecture} of Haglund, Remmel, and Wilson \cite{HRW} depends
on two positive integers $k \leq n$ and predicts the equality of three formal
power series in an infinite set of variables $\xx = (x_1, x_2, \dots )$ and
two additional parameters $q$ and $t$:
\begin{equation}
    \Delta'_{e_{k-1}} e_n = \mathrm{Rise}_{n,k}(\xx;q,t) = \mathrm{Val}_{n,k}(\xx;q,t).
\end{equation}
Here $\Delta'_{e_{k-1}}$ is a Macdonald eigenoperator and $\mathrm{Rise}$ and 
$\mathrm{Val}$ are defined in terms of lattice path combinatorics; see \cite{HRW}
for details.

Although the Delta Conjecture is open in general, 
it is proven when one of the parameters
$q,t$ is set to zero. Combining results of 
\cite{GHRY, HRW, HRS, HRS2, RhoadesOSP, WMultiset} we have
\begin{equation}
    \label{delta-zero}
    \Delta'_{e_{k-1}} e_n \mid_{t = 0} = 
    \mathrm{Rise}_{n,k}(\xx;q,0) = \mathrm{Rise}_{n,k}(\xx;0,q) = 
    \mathrm{Val}_{n,k}(\xx;q,0) = \mathrm{Val}_{n,k}(\xx;0,q).
\end{equation}
If $C_{n,k}(\xx;q)$ is the common symmetric function in Equation~\eqref{delta-zero},
we have \cite{HRS}
\begin{equation}
    \grFrob(R_{n,k};q) = (\mathrm{rev}_q \circ \omega) C_{n,k}(\xx;q),
\end{equation}
where $\mathrm{rev}_q$ reverses the coefficient sequences of polynomials in $q$
and $\omega$ is the symmetric function involution which trades $e_n$ and $h_n$,
so that $R_{n,k}$ gives a representation-theoretic model for the Delta Conjecture
at $t = 0$. 

The rings $R_{n,k}$ also have a geometric interpretation. For $k \leq n$, 
Pawlowski and Rhoades \cite{PR} introduced the variety
\begin{equation}
    X_{n,k} := \{ (\ell_1, \dots, \ell_n) \,:\, 
    \text{$\ell_i$ a line in $\CC^k$ and } \ell_1 + \cdots + \ell_n = \CC^k \}
\end{equation}
of $n$-tuples of $1$-dimensional subspaces of $\CC^k$ which have full span.
They proved \cite{PR} that the rational cohomology of $X_{n,k}$ is presented
by the ring $R_{n,k}.$
Rhoades and Wilson \cite{RW} gave another interpretation of $R_{n,k}$ using an
extension of the Vandermonde determinant to superspace.

\subsection{The rings $R_\mu$}

The \defn{Garsia-Procesi modules} $R_\mu$, 
indexed by partitions $\mu \vdash n$, are another generalization of the coinvariant ring defined by $$R_\mu=\mathbb{Q}[x_1,\ldots,x_n]/I_\mu$$ where we define $I_\mu$ using the notation of Garsia and Procesi \cite{GarsiaProcesi} as follows.   For a subset $S\subseteq\{x_1,\ldots,x_n\}$, define the \textit{partial elementary symmetric functions} $e_r(S)$ to be the elementary symmetric function of degree $r$ in the restricted set of variables $S$.  For instance, $e_2(x_1,x_4,x_5)=x_1x_4+x_1x_5+x_4x_5$.  

Let $\mu'$ be the conjugate partition formed by reflecting $\mu$ about the diagonal, and define \begin{equation}\label{eq:stmu}
    c_t(\mu)=\mu'_1+\cdots+\mu'_t-t
\end{equation} to be the number of squares in the first $t$ columns that lie above the first row.  Then we have\footnote{It is straightforward to verify that the inequality in (\ref{eq:inequality}) is equivalent to the one stated in \cite{GarsiaProcesi}, and we omit the proof for brevity.} \begin{equation}\label{eq:inequality}
    I_\mu=\langle e_r(S)\,:\,c_{n-|S|}(\mu)<r\le |S|\rangle.
\end{equation}
  Note that in the case $\mu=(1^n)$, we recover the coinvariant ring, that is, $R_{(1^n)}=R_n$.  In general, the graded Frobenius character of $R_\mu$ is given by  $$\grFrob(R_\mu;q)=\widetilde{H}_\mu(\xx;q)$$ where $\widetilde{H}_{\mu}(\xx;q)$ are the classical Hall-Littlewood polynomials.  These exhibit a combinatorial formula in terms of the following notions.  
  \begin{definition}\label{def:SSYT}
    A \textbf{semistandard Young tableau} $T$ of shape $\lambda$ is a filling of the Young diagram of $\lambda$ with positive integers such that the rows are weakly increasing left to right and the columns are strictly increasing bottom to top.  The \textbf{content} of a tableau $T$ (or word $w$) is the tuple $(m_1,m_2,\ldots)$ where $m_i$ is the number of times $i$ appears in $T$ (or $w$).  
  \end{definition}
    
    Write $\SSYT(\lambda,\mu)$ for the set of all semistandard Young tableaux of shape $\lambda$ and content $\mu$.  Then it was shown in \cite{LascouxSchutzenberger} that $$\widetilde{H}_\mu(\xx;q)=\sum_{\lambda}\sum_{S\in \SSYT(\lambda,\mu)}q^{\cc(S)}s_{\lambda}$$ where $\cc$ is a generalization of the cocharge statistic that we describe in detail in section \ref{sec:Rmu}. 

 The rings $R_\mu$ also have a geometric interpretation in terms of \textit{Springer fibers}.  Define $\mathcal{B}_\mu$ to be the subvariety of the full flag variety $$\mathrm{Fl}_n=\{0\subseteq V_1\subseteq V_2 \subseteq \cdots \subseteq V_n=\mathbb{C}^n \,:\, \dim(V_i)=i\text{ for all }i\}$$ consisting of the flags fixed by the action of a fixed unipotent element $u$ of $\mathrm{GL}_n(\mathbb{C})$ having Jordan blocks of size $\mu_1,\mu_2,\ldots,\mu_{\ell(\mu)}$.  The space $\mathcal{B}_\mu$ is a fiber of the Springer resolution of the unipotent subvariety of $\mathrm{GL}_n$, and its cohomology ring comes with a graded $S_n$-module structure whose top degree component is precisely the irreducible representation $V_\mu$ \cite{Springer}. The work of \cite{deConciniProcesi} and \cite{Tanisaki} shows that $R_\mu$ is isomorphic to the cohomology ring of the Springer fiber $\mathcal{B}_\mu$, both as a graded ring and as a graded $S_n$-module.

\subsection{The rings $R_{n,k,\mu}$}

In \cite{Griffin}, Griffin introduced a common generalization of $R_\mu$ and $R_{n,k}$.  While Griffin's notation for these generalized modules is $R_{n,\lambda,s}$, here we change the variable $s$ to $k$ and $\lambda$ to $\mu$ and interchange their order to instead write $R_{n,k,\mu}$.  This notation is more compatible with the way we denoted the two known modules above.

Griffin defines the ideal $I_{n,k,\mu}$ generated by: \begin{itemize}
    \item The monomials $x_1^k,\ldots,x_n^k$, and
    \item The partial elementary symmetric functions $e_r(S)$ satisfying $$c_{n-|S|}(\mu)+(n-|\mu|)<r\le |S|,$$ where the notation $c_t(\mu)$ is the same as in Equation (\ref{eq:stmu}).
\end{itemize} 
  Then we have $$R_{n,k,\mu}=\mathbb{C}[x_1,\ldots,x_n]/I_{n,k,\mu}.$$
Notice that if $|\mu|=n$ and $k\ge \ell(\mu)$, then $R_{n,k,\mu}=R_\mu$, and if $\mu=(1^k)$ then $R_{n,k,\mu}=R_{n,k}$.

  In \cite{Griffin}, Griffin gives several combinatorial formulas for the graded Frobenius series of $R_{n,k,\mu}$.  The most relevant of these to our purposes is an expansion in terms of Hall-Littlewood polynomials.  In the following, we write $$H_{\lambda}(x;q):=q^{n(\lambda)}\widetilde{H}(x;q^{-1})=\rev_q(\widetilde{H}(x;q))$$ to denote the `charge' Hall-Littlewood polynomials, where $n(\lambda)=\sum_i \binom{\lambda_i'}{2}$ for any partition $\lambda$.  With this notation, we have

\begin{equation}\label{eqn:FrobRnkmu}
    \grFrob(R_{n,k,\mu};q)=\mathrm{rev}_q\left(\sum_{\substack{\lambda\supseteq \mu \\ \ell(\lambda)\le k \\ |\lambda|=n}} q^{n(\lambda,\mu)}\prod_{i\ge 0} \binom{\lambda_{i}'-\mu_{i+1}'}{\lambda_{i}'-\lambda_{i+1}'}_q H_\lambda(x;q) \right)
\end{equation} 
where $n(\lambda,\mu)=\sum_i \binom{\lambda_i'-\mu_i'}{2}$ and where the notation $\binom{a}{b}_q$ denotes the $q$-binomial coefficient $\prod_{i=0}^{b-1} \frac{1-q^{a-i}}{1-q^{b-i}}$.  The notation $\lambda\supseteq \mu$ indicates that the Young diagram of $\mu$ is contained inside that of $\lambda$.

The modules $R_{n,k,\mu}$ have a geometric interpretation as well, in the limit as $k\to \infty$.   The \textit{Eisenbud-Saltman rank variety} $\overline{O}_{n,\mu}$ is the subvariety of $\mathfrak{gl}_n$ defined by 
$$\overline{O}_{n,\mu}=\{X\in \mathfrak{gl}_n: \dim \ker X^d\ge \mu_1'+\cdots + \mu_d',d=1,\ldots,n\}.$$ In the case that $|\mu|=n$, this coincides with the closure of the variety $O_{\mu}$ of nilpotent matrices with Jordan block type $\mu$.  Setting $R_{n,\mu}$ to be the limiting module of $R_{n,k,\mu}$ as $k\to \infty$, Griffin shows that $R_{n,\mu}$ is the coordinate ring of the scheme theoretic intersection $$\overline{O}_{n,\mu'}\cap \mathfrak{t}$$ where $\mathfrak{t}$ is the Cartan subalgebra of diagonal matrices in $\mathfrak{gl}_n$.  This is a strict generalization of the analogous result for $R_\mu$ and $\overline{O}_{\mu}$, which was an essential step in de Concini and Procesi's work \cite{deConciniProcesi} on the connections to the Springer fibers.

\subsection{Main results}

Our main results are as follows.

\begin{theorem}
\label{thm:Rnk}
Let $k \leq n$ be positive integers. Consider the set of polynomials
$$
\BBB_{n,k} := \{ F_T^S \cdot e_1^{i_1} e_2^{i_2} \cdots e_{n-k}^{i_{n-k}} \},
$$
where $T, S \in \SYT(n)$ have the same shape and
$(i_1, i_2, \dots, i_{n-k})$ is a tuple of $n-k$ nonnegative integers whose sum
is $< k - \des(S)$. The set $\BBB_{n,k}$ descends to a higher Specht basis for $R_{n,k}$.
\end{theorem}

More details on the notation above, as well as the proof, can be found in Section \ref{sec:Rnk}.  For now, note that since $\symm_n$ acts trivially on any elementary symmetric polynomial, Theorem~\ref{thm:Rnk}
immediately implies \cite[Cor. 6.13]{HRS}:
$$
\grFrob(R_{n,k}; q) = \sum_{S \in \SYT(n)} q^{\maj(S)} \binom{n - \des(S) - 1}{n-k}_q s_{\shape(S)}.
$$

For $R_\mu$, we use a generalization of the cocharge statistic to define \textit{semistandard higher Specht polynomials} $F_T^S$ where $S$ is a semistandard tableau with content $\mu$ and $T$ is a standard Young tableau of the same shape as $S$.   (See Section \ref{sec:Rmu}.)  The polynomial $F_T^S$ is homogeneous of degree $\cc(S)$. 

\begin{conjecture}\label{conj:R-mu}
 Let $\mu\vdash n$. Consider the set of semistandard higher Specht polynomials $$\mathcal{B}_\mu=\{F_T^S\}$$ for which $S$ has content $\mu$ and $T\in \SYT(n)$ has the same shape as $S$.  Then $\mathcal{B}_\mu$ descends to a higher Specht basis of $R_\mu$.
\end{conjecture}

  Numerically, the conjectured basis matches what we would expect based on the graded Frobenius character of $R_\mu$ (as computed in \cite{LascouxSchutzenberger}), which is given by $$\grFrob(R_\mu;q)=\sum_{\lambda}\sum_{S\in \SSYT(\lambda,\mu)}q^{\cc(S)}s_{\lambda}.$$ 
Our main progress towards proving this conjecture is the following.
\begin{theorem}\label{thm:two-rows}
  Conjecture \ref{conj:R-mu} holds when $\mu=(k,n-k)$ has two rows.
\end{theorem}

Finally, we combine these two results to give a higher Specht basis for an infinite family of the modules $R_{n,k,\mu}$, as follows.

\begin{theorem}\label{thm:Rnkmu}
  Suppose $\mu$ is the one-row partition $(n-1)$.  Consider the set of polynomials $$\mathcal{B}_{n,k,(n-1)}=\{F_T^S\cdot e_1^i\}$$ where $F_T^S\in \mathcal{B}_{(n-1,1)}$ is a semistandard higher Specht polynomial for the shape $(n-1,1)$, and $i<k-\des(S)$.  Then $\mathcal{B}_{n,k,(n-1)}$ descends to a higher Specht basis of $R_{n,k,(n-1)}$. 
\end{theorem}

We can see numerically that the basis of Theorem \ref{thm:Rnkmu} matches what we would expect from the graded Frobenius character.  In particular, setting $\mu=(n-1)$ in Equation (\ref{eqn:FrobRnkmu}), the summation has two terms, with $\lambda=(n)$ and $\lambda=(n-1,1)$.  In both cases we have $n(\lambda,\mu)=0$, and the only nontrivial $q$-binomial coefficient occurs at $i=0$ with $\lambda=(n-1,1)$.  Hence
\begin{align*}
    \grFrob(R_{n,k,(n-1)};q) &=\rev_q\left[ H_{(n)}(x;q)+\binom{k-1}{k-2}_{q} H_{(n-1,1)}(x;q)\right] \\
    &=\rev_q\left[H_{(n)}(x;q)+(1+q+\cdots+q^{k-2})H_{(n-1,1)}(x;q)\right] \\
    &=q^{k-1}\left[H_{(n)}(x;q^{-1})+(1+q^{-1}+q^{-2}+\cdots +q^{2-k})H_{(n-1,1)}(x;q^{-1})\right] \\
    &=q^{k-1}\widetilde{H}_{(n)}(x;q)+(1+q+q^2+\cdots+q^{k-2})\widetilde{H}_{(n-1,1)}(x;q)
\end{align*}

where the third equality above follows from the fact that $H_{(n)}(x;q)$ has degree $0$ in $q$ and $H_{(n-1,1)}$ has degree $1$, so that the entire polynomial has degree $k-1$.  Finally, note that $\widetilde{H}_{(n)}(x;q)$ and $\widetilde{H}_{(n-1,1)}(x;q)$ are the Frobenius series of the Garsia-Procesi modules $R_{(n)}$ and $R_{(n-1,1)}$ respectively.  It follows from Theorem \ref{thm:two-rows} that the basis $\mathcal{B}_{n,k,(n-1)}$ of Theorem \ref{thm:Rnkmu} gives the correct number of irreducible $S_n$ representations in each degree.

It is our hope that these methods can be generalized to construct a higher Specht basis for $R_{n,k,\mu}$ of the form $\{F_T^S\cdot e_1^{i_1}\cdots e_{n-|\mu|}^{i_{n-|\mu|}}\}$, where the polynomials $F_T^S$ are semistandard higher Specht polynomials for various partitions $\lambda\vdash n$ such that $\lambda\supset \mu$, and where there is an appropriate bound on the exponents $i_j$.  As it is, one current limitation is that for any partition $\mu$ with $|\mu|<n$ and $\mu\neq (n-1)$, there exists a  partition $\lambda$ of $n$ containing $\mu$ that has at least three rows.  This exceeds the two-row condition of Theorem \ref{thm:two-rows}.

\subsection{Outline}

The remainder of the paper is organized as follows.  In Section 2 we prove Theorem \ref{thm:Rnk}.  In Section 3 we prove Theorem \ref{thm:two-rows} and provide additional evidence and work towards Conjecture \ref{conj:R-mu}.  Finally, we prove Theorem \ref{thm:Rnkmu} in section 4. 

\subsection{Acknowledgments}

We thank Nicolas Thi\'{e}ry for several illuminating conversations about higher Specht polynomials, and for sharing some helpful code.  Thanks also to Sean Griffin for many discussions regarding his new $S_n$-modules $R_{n,k,\mu}$.

We used the open source mathematics software Sage extensively in testing conjectures and generating examples for this work.

\section{Higher Specht bases for $R_{n,k}$}\label{sec:Rnk}

We will obtain our new basis for $R_{n,k}$ by multiplying the 
higher Specht polynomials $F_T^S$ (for standard tableaux $T$ and $S$ of the same shape) by appropriate elementary symmetric polynomials.
Before stating our basis, we recall some notions from 
commutative algebra.

A sequence of polynomials $f_1, f_2, \dots, f_r$ in the ring 
$\QQ[x_1, \dots, x_n]$ is a \defn{regular sequence} if for all $1 \leq j \leq r$,
the endomorphism
\begin{equation*}
    \QQ[x_1, \ldots, x_n] / ( f_1, \ldots, f_{j-1})
    \xrightarrow{\times f_j}
    \QQ[x_1, \ldots, x_n] / ( f_1, \ldots, f_{j-1})
\end{equation*}
on the quotient ring $\QQ[x_1, \ldots, x_n] / (f_1, \ldots, f_{j-1})$
induced by multiplication by $f_j$ is injective.
The longest possible length of a regular sequence $f_1, \ldots, f_r$ in 
$\QQ[x_1, \ldots, x_n]$ is $r = n$. The elementary symmetric polynomials
$e_1, e_2, \ldots, e_n$ constitute one such length $n$ regular sequence.

Let $f_1, \ldots, f_n$ be any length $n$ regular sequence in $\QQ[x_1, \ldots x_n]$
such that the $f_j$ are homogeneous. Then the quotient
$\QQ[x_1, \dots, x_n]/( f_1, \dots, f_n)$ is graded and if 
$\mathcal{B}$ is a family of homogeneous polynomials which descends to a $\QQ$-basis
of $\QQ[x_1, \dots, x_n]/(f_1, \dots, f_n)$, then
the infinite set of polynomials
\begin{equation*}
    \{ g \cdot f_1^{i_1} f_2^{i_2} \cdots f_n^{i_n} \,:\, g \in \mathcal{B} \}
\end{equation*}
is a basis of the full polynomial ring $\QQ[x_1, \dots, x_n]$.
In order to describe our basis of $R_{n,k}$ we need one more definition.

\begin{definition}\label{def:majdes}
  A \defn{descent} of a standard Young tableau $S$ is an entry $i$ which occurs in a lower row than $i+1$ (written in French notation).  The \defn{major index} of $S$, written $\maj(S)$, is the sum of the descents of $S$, and we write $\des(S)$ for the number of descents.
\end{definition}

For instance, if $S$ is the tableau at left in Figure \ref{fig:SYT}, then $\maj(S)=1+3+4+6=14$ and $\des(S)=4$.

We now restate Theorem \ref{thm:Rnk} here for the reader's convenience.
{
\renewcommand{\thetheorem}{\ref{thm:Rnk}}
\begin{theorem}
Let $k \leq n$ be positive integers. Consider the set 
$$
\BBB_{n,k} := \{ F_T^S \cdot e_1^{i_1} e_2^{i_2} \cdots e_{n-k}^{i_{n-k}} \},
$$
consisting of all polynomials of the form $F_T^S \cdot e_1^{i_1} e_2^{i_2} \cdots e_{n-k}^{i_{n-k}}$ where $S, T \in \SYT(n)$ have the same shape and
$(i_1, i_2, \dots, i_{n-k})$ is a tuple of $n-k$ nonnegative integers whose sum
is $< k - \des(S)$. The set $\BBB_{n,k}$ descends to a higher Specht basis for $R_{n,k}$.
\end{theorem}
\addtocounter{theorem}{-1}
}

We first prove an enumerative lemma which will help us in the proof of Theorem \ref{thm:Rnk}.

\begin{lemma}
  \label{enumeration-lemma}
  Let $n$ and $k$ be positive integers.
  There are exactly $k^n$  tuples $(S,T,i_1, \dots, i_n)$ where $S$ and $T$ are standard Young tableaux of 
  the same shape with $n$ boxes and $i_1, \dots, i_n$ are nonnegative integers with 
  $i_1 + \cdots + i_n < k - des(S)$.
\end{lemma}

\begin{proof}
 The \textit{Robinson-Schensted-Knuth correspondence} gives a bijection between words $w\in \{1,2,\ldots,k\}^n$ and pairs $(R,T)$ of Young tableaux with $n$ boxes having the same shape, such that $R$ is semistandard with entries in $\{1,2,\ldots,k\}$ and $T$ is standard.  
 
 Since there are $k^n$ words $w\in \{1,2,\ldots,k\}^n$, it suffices to give a bijection between  the tuples $(S,T,i_1,\ldots,i_n)$ in question and the pairs $(R,T)$ described above. Given a standard Young tableau $S$, define the \textit{destandardization} of $S$, denoted $S'$ as follows.  If $d_1<d_2<\cdots <d_n$ are the descents of $S$, replace $1,\ldots,d_1$ with $1$, replace $d_1+1,\ldots,d_2$ with $2$, etc.  See Figure \ref{fig:destandardization} for an example.
 
 \begin{figure}
     \centering
     $\young(6,378,1245)$\hspace{1.5cm} $\young(3,233,1122)$\hspace{1.5cm}$\young(6,377,1235)$
     \caption{A standard Young tableau $S$ (at left) and its destandardization $S'$ (middle).  The tableau $R$ defined from $S'$ in the proof of Lemma \ref{enumeration-lemma} arising from the tuple $(i_1,\ldots,i_8)=(0,1,0,0,2,0,1,0)$ is shown at right.}
     \label{fig:destandardization}
 \end{figure}
  Note that $S'$ is semistandard by the definition of a descent, and $S$ can be uniquely reconstructed from $S'$.  Notice also that the largest entry of $S'$ is $\des(S)+1$.
  
  Now let $i_1,\ldots,i_n$ be such that $i_1+\cdots + i_n<k-\des(S)$.  Let $a_1\le \cdots \le a_n$ be the entries of $S'$ in order (breaking ties by the corresponding ordering in $S$).  Then define $R$ by increasing each of the numbers $a_{1},\ldots,a_n$ by $i_1$, then increasing $a_2,\ldots,a_n$ by $i_2$, then increasing $a_3,\ldots ,a_n$ by $i_3$, and so on. Because $i_1+\cdots +i_n<k-\des(S)$, the result $R$ has largest entry at most $k$. This process is reversible, and so the proof is complete.
\end{proof}

We now prove Theorem \ref{thm:Rnk}.

\begin{proof} (of Theorem \ref{thm:Rnk})
It will be convenient to consider a broader family of quotients $R_{n,k,s} = \QQ[x_1, \dots, x_n]/I_{n,k,s}$ 
defined for $s \leq k \leq n$. Here
\begin{equation}
I_{n,k,s} := \langle x_1^k, x_2^k, \dots, x_n^k, e_n, e_{n-1}, \dots, e_{n-s+1} \rangle.
\end{equation}
In particular, we have $I_{n,k,k} = I_{n,k}$ and $R_{n,k,k} = R_{n,k}$.
We allow $s$ to be zero, in which case no $e$'s appear in our ideal at all. 
However, we assume that $n, k$ are positive.

Consider the following extended set 
\begin{equation}
\BBB_{n,k,s} :=  \{ F_T^S \cdot e_1^{i_1} e_2^{i_2} \cdots e_{n-s}^{i_{n-s}} \},
\end{equation}
consisting of all polynomials $F_T^S \cdot e_1^{i_1} e_2^{i_2} \cdots e_{n-s}^{i_{n-s}}$ for which $S, T \in \SYT(n)$ have the same shape and $(i_1, i_2, \dots, i_{n-s})$ is a list of $n-s$ nonnegative integers
whose sum is $< k - \des(S)$.

We claim that $\BBB_{n,k,s}$ descends to a basis for $R_{n,k,s}$ for all $n, k, s$.  This is stronger than the statement
of the Theorem.  When $n = k = s$, the quotient $R_{n,n,n}$ is the classical coinvariant algebra and the fact that $\BBB_{n,n,n}$
descends to a basis for $R_{n}$ is precisely the result of \cite{ATY}.

To begin, the proof of \cite[Lem. 6.9]{HRS} gives a short exact sequence
\begin{equation}\label{eq:SES}
0 \rightarrow R_{n,k-1,s} \rightarrow R_{n,k,s} \rightarrow R_{n,k,s+1} \rightarrow 0,
\end{equation}
where the first map is induced by multiplication by $e_{n-s}$ and the second map is the canonical projection.
(In fact, the \cite[Lem. 6.9]{HRS} 
is only proven in the case where $s > 0$; the case $s = 0$ has the same proof after observing that
$\dim(R_{n,k,0}) = k^n$, so that the dimensions of the rings on either end add up to the dimension of the ring in the 
middle in this case.)

By exactness, if $\BBB$ descends to a basis for $R_{n,k-1,s}$ and if $\CCC$ descends 
to a basis for $R_{n,k,s+1}$ the disjoint union
\begin{equation}
\{ e_{n-s} \cdot f \in \BBB \} \sqcup \{ g \,:\, g \in \CCC \}
\end{equation}
descends to a basis for $R_{n,k,s}$.   Using this property and the fact that 
\begin{equation}
\{e_{n-s} \cdot f \in \BBB_{n,k-1,s} \} \sqcup \{ g \,:\, g \in \BBB_{n,k,s+1} \} = \BBB_{n,k,s},
\end{equation}
we are inductively reduced to proving the result when $s = 0$.  That is, it remains to show that 
$\BBB_{n,k,0}$ descends to a basis for $R_{n,k,0}$.

By definition, we have
\begin{equation}
\BBB_{n,k,0} =  \{ F_T^S \cdot e_1^{i_1} e_2^{i_2} \cdots e_{n}^{i_{n}} \},
\end{equation}
where $S, T \in \SYT(n)$ have the same shape and $(i_1, \dots, i_n)$ is a sequence of nonnegative integers 
whose sum is $< k - \des(S)$.  By the definition of the cocharge word $\cw(S)$, the largest possible exponent appearing in
the monomial
$\xx_T^{\cw(S)}$ or the polynomial $F_T^S = \varepsilon_T \cdot \xx_T^{\cw(S)}$ is $\des(S)$.
Since elementary symmetric polynomials are sums of squarefree monomials, we see that the largest possible exponent
appearing in a polynomial in $\BBB_{n,k,0}$ is $k-1$.  Since 
\begin{equation}
R_{n,k,0} = \QQ[x_1, \dots, x_n]/\langle x_1^k, \dots, x_n^k\rangle
\end{equation}
and $|\BBB_{n,k,0}| = k^n = \dim(R_{n,k,0})$, 
(where the first equality uses Lemma~\ref{enumeration-lemma})
we conclude that 
$\BBB_{n,k,0}$ descends to a basis for $R_{n,k,0}$ if and only if $\BBB_{n,k,0}$ is linearly independent
in the full polynomial ring $\QQ[x_1, \dots, x_n]$.

We finish the proof by showing that  
$\BBB_{n,k,0}$ is linearly independent in $\QQ[x_1, \dots, x_n]$.
To do this, we apply the main result of \cite{ATY}, that the set 
$$\BBB_n = \{ F_T^S \,:\, \text{$S,T \in \SYT(n)$ have the same shape} \}$$ descends to a basis for the 
coinvariant ring $R_n$.  Since the ideal defining $R_n$ is cut out by the regular sequence $e_1, \dots, e_n$,
we know that the set 
\begin{equation}
\{ F_T^S \cdot e_1^{i_1} \cdots e_n^{i_n} \,:\, 
\text{$S, T \in \SYT(n)$ have the same shape and $i_1, \dots, i_n \geq 0$} \}
\end{equation}
is a basis of the full polynomial ring $\QQ[x_1, \dots, x_n]$.
Since it is a subset of this basis, the set $\BBB_{n,k,0}$ is linearly independent in $\QQ[x_1, \dots, x_n]$, as desired.
\end{proof}

\section{Higher Specht bases for $R_\mu$}\label{sec:Rmu}

We now give a conjectured generalization of the higher Specht basis to the Garsia-Procesi modules $R_\mu$, and prove it in the case that $\mu$ has at most two rows.  We first recall the generalization of cocharge, defined in \cite{LascouxSchutzenberger}, to words whose \textit{content} (Definition \ref{def:SSYT}) is a partition.  Throughout this section we assume $w$ is a word with partition content $\mu$.

For an entry $w_j$ of $w$ and a positive integer $k$, define the \defn{cyclically previous} $k$ before $w_j$, denoted  $\mathrm{cprev}(k,w_j)$,  to be the rightmost $k$ cyclically to the left of $w_j$ in $w$.  That is, it is the rightmost $k$ to the left of $w_j$ if such a $k$ exists, or the rightmost $k$ in $w$ otherwise.

\begin{definition}
  Let $w_{i_1}=1$ be the rightmost $1$ in $w$, and recursively define $i_2,\ldots,i_{\ell(\mu)}$ by $$w_{i_{j+1}}=\mathrm{cprev}(j+1,w_{i_j}).$$  We call the subword $w^{(1)}$ consisting of the entries $w_{i_j}$ the \textbf{first standard subword} of $w$.
\end{definition}

\begin{definition}
  The \textbf{standard subword decomposition} of $w$ is obtained by setting $w^{(1)}$ to be the first standard subword of $w$, and recursively defining $w^{(i)}$, for $i>1$, to be the first standard subword of the entries of $w$ not in $w^{(1)},\ldots,w^{(i-1)}$.
\end{definition}

\begin{definition}
The \textbf{cocharge} of $w$ is $$\cc(w)=\sum_{i}\cc(w^{(i)})$$ where $w^{(1)},w^{(2)},\ldots,w^{(\mu_1)}$ is its standard subword decomposition. The \defn{cocharge word} $\cw(w)$ is defined as the labeling on $w$ given by labeling the letters of $w^{(i)}$ with its cocharge word $\cw(w^{(i)})$ for each $i$.  

 For a semistandard Young tableau $S$ having reading word $w$, we define $$\cc(S)=\cc(w).$$  For a square $s$ in the diagram of $S$, we write $\cw_S(s)$ for the cocharge word label of the corresponding letter of $w$.
\end{definition}

\begin{example}
  The semistandard Young tableau $$S=\young(4,22334,11123)$$ has reading word $w=42233411123$.  If we label the first standard subword $w^{(1)}$ (shown in boldface below) with its cocharge labeling as subscripts, we get: $$4_{\phantom{0}} 2_{\phantom{0}} \mathbf{2}_1 3_{\phantom{0}} 3_{\phantom{0}} \mathbf{4}_2 1_{\phantom{0}} 1_{\phantom{0}} \mathbf{1}_0 2_{\phantom{0}} \mathbf{3}_1.$$ Then we label $w^{(2)}$ to obtain: $$\mathbf{4}_2 \mathbf{2}_1 2_1 3_{\phantom{0}} \mathbf{3}_1 4_2 1_{\phantom{0}} \mathbf{1}_0 1_0 2_{\phantom{0}} 3_1.$$ We finally label $w^{(3)}$ to obtain:
  $$4_2 2_1 2_1 \mathbf{3}_1 3_1 4_2 \mathbf{1}_0 1_0 1_0 \mathbf{2}_0 3_1.$$ It follows that $\cw(w)=2111200001$ and $\cc(S)=\cc(w)=8$.
\end{example}

We can now define the conjectured basis for $R_\mu$. 

\begin{definition}
  Let $(S,T)$ be a pair of Young tableaux of the same shape $\lambda\vdash n$ where $S$ is semistandard and has content $\mu$ and $T$ has content $(1^n)$ (but is not necessarily standard).  Then we define $$\mathbf{x}_T^S=\prod_{s\in D(\lambda)} x_{T(s)}^{\cw_S(s)}$$ where $D(\lambda)$ is the set of squares in the diagram of $\lambda$.  Finally, define the \defn{semistandard higher Specht polynomial}
  $$F_T^S=\varepsilon_T \cdot \mathbf{x}_T^S.$$
\end{definition}

Recall that $\SSYT(\lambda,\mu)$ is the set of all semistandard Young tableaux of shape $\lambda$ and content $\mu$.  We also write $\SYT(\lambda)=\SSYT(\lambda, (1^n))$ for the set of standard Young tableaux of shape $\lambda$. Then we can restate Conjecture \ref{conj:R-mu} as follows.

\begin{conjecture}[Conjecture \ref{conj:R-mu} restated]
 The set of polynomials $$\BBB_\mu=\left\{F_T^S \,:\, (S,T)\in \bigcup_{\lambda\vdash n} \SSYT(\lambda,\mu)\times \SYT(\lambda) \right\}$$ is a basis of $R_\mu$.
\end{conjecture}

\subsection{Semistandard higher Specht modules in $\QQ[\xx_n]$}

As a step towards proving Conjecture \ref{conj:R-mu}, we consider the modules generated by the semistandard higher Specht polynomials as submodules of the full polynomial ring $\QQ[\xx_n]$, before descending to the quotient $R_\mu$.  In particular, we show that these give copies of the ordinary polynomial Specht modules in higher degrees.

\begin{definition}
Write $\Tab(\lambda)$ to denote the set of all (not necessarily standard) tableaux of shape $\lambda$ and of content $(1^n)$.  In other words, $\Tab(\lambda)$ is the set of all $n!$ ways of filling the boxes of $\lambda$ with the numbers $1,2,3,\ldots,n$ in any manner.

Note that if $\lambda\vdash n$ then $S_n$ naturally acts on $\Tab(\lambda)$ by permuting entries in a tableau.
\end{definition}

\begin{definition}
For a fixed $S\in \SSYT(\lambda,\mu)$, define $$V^S:=\mathrm{span}\{F^S_T:T\in \mathrm{Tab}(\lambda)\}$$ to be the span of the higher Specht polynomials associated to $S$, considered as a subspace of $R=\mathbb{C}[x_1,\ldots,x_n]$ where $n=|\lambda|$.  Similarly define $\overline{V^S}$ to be its image in the quotient $R_\mu$.  
\end{definition}

We first show that $V^S$ is an irreducible $S_n$-module isomorphic to the standard Specht module $V^\lambda$.  We begin with several technical lemmas.  Throughout, we fix a choice of semistandard Young tableau $S\in \SSYT(\lambda,\mu)$.

\begin{proposition}
  Let $\omega\in S_n$ and $T\in \Tab(\lambda)$.  Then $$\omega F_T^S=F_{\omega T}^S.$$ 
\end{proposition}

\begin{proof}
  First note that if $\tau \in C(T)$ then $\tau':=\omega \tau \omega^{-1}\in C(\omega T)$, and similarly if $\sigma \in R(T)$ then $\sigma':=\omega\sigma\omega^{-1}\in R(\omega T)$.  Notice also that $\omega \mathbf{x}_T^S=\mathbf{x}_{\omega T}^S$.
  We therefore have 
  \begin{align*}
    \omega F_T^S=\omega \varepsilon_T \mathbf{x}_T^{S} 
    &= \sum_{\tau\in C(T)} \sum_{\sigma \in R(T)} \sgn(\tau)\omega \tau \sigma \mathbf{x}_T^{S} \\
    &= \sum_{\tau\in C(T)} \sum_{\sigma \in R(T)} \sgn(\omega \tau\omega^{-1})(\omega \tau \omega^{-1})(\omega \sigma \omega^{-1})\omega \mathbf{x}_T^{S} \\
    &= \sum_{\tau'\in C(\omega T)} \sum_{\sigma' \in R(\omega T)} \sgn(\tau')\tau'\sigma'\mathbf{x}_{\omega T}^{S} \\
    &= F_{\omega T}^S
  \end{align*}
  as desired.
\end{proof}

\begin{corollary}
  The space $V^S$ is a cyclic $S_n$-submodule of $R$.
\end{corollary}

We now show that, assuming the polynomials $F_T^S$ are independent for $T$ standard, the submodule $V^S$ is a copy of the irreducible $S_n$-module $V^\lambda$.  We recall (see, for instance, \cite{Peel}) the Garnir relations that govern the $S_n$-module structure of $V^\lambda$ with respect to the standard Specht basis.

\begin{definition}
  Let $T\in \Tab(\lambda)$.  Let $a$ and $b$, with $a<b$, be the indices of two distinct columns of $T$, and let $t\le \lambda'_b$ be a row index of one of the entries of column $b$.  Then we write $S^{a,b}_t$ to be the subgroup of $S_n$ consisting of all permutations of the set of elements of $T$ residing either in column $a$ weakly above $t$, or in column $b$ weakly below $t$. 
  
  The \defn{Garnir element} $G^{a,b}_t$ is the partial antisymmetrizer $$G^{a,b}_t:=\sum_{\omega\in S^{a,b}_t} \sgn(\omega)\omega.$$
\end{definition}

\begin{proposition}\label{prop:Garnir}
  The element $F_T^S$, for any $T\in \Tab(\lambda)$, satisfies the Garnir relation $G^{a,b}_{t}(F_T^S)=0$.
\end{proposition}

To prove this proposition, we first show that the analog of Lemma 3.3 in \cite{Peel} holds here.  To state it, we introduce the Young (anti)symmetrizers $\alpha$ and $\beta$ defined as follows.  For any subgroup $U\subseteq S_n$, define $$\alpha(U)=\sum_{\tau \in U}\sgn(\tau)\tau\hspace{0.5cm}\text{ and }  \hspace{0.5cm}\beta(U)=\sum_{\sigma\in U}\sigma.$$  In this notation, the Young symmetrizer $\varepsilon_T$ can be written as $$\varepsilon_T=\alpha(C(T))\beta(R(T)).$$

\begin{lemma}
  Let $U$ be any subgroup of $S_n$ and let $C=C(T)$ where $T\in \Tab(\lambda)$.  Suppose there is an involution $\sigma\mapsto \sigma'$ on $UC$ such that for each $\sigma\in UC$, there exists $\rho_\sigma\in R(T)$ for which $\rho_\sigma^2=1$, $\sgn(\rho_\sigma)=-1$, and $\sigma'=\sigma\rho_\sigma$.  Then $$\alpha(U) F_T^S=0.$$
\end{lemma}

\begin{proof}
  We have $\alpha(U)\alpha(C)=|U\cap C|\alpha(UC)$ (see Lemma 3.2 in \cite{Peel}).  Therefore \begin{align*}
      \alpha(U)F_T^S&=\alpha(U)\varepsilon_T \mathbf{x}_T^S \\
      &= \alpha(U)\alpha(C)\beta(R)\mathbf{x}_T^S \\
      &= |U\cap C| \alpha(UC)\beta(R) \mathbf{x}_T^S
  \end{align*}
  where $R=R(T)$ is the group of row permutations.  Due to the involution $\sigma\mapsto \sigma'$, which has no fixed points because $\sgn(\rho_\sigma)=-1$ for each $\rho_\sigma$, we have that the terms in $\alpha(UC)$ can be partitioned into pairs of terms $\sgn(\sigma)\sigma+\sgn(\sigma')\sigma'$.  We claim that each of these two-term sums kills $\beta(R)x_T^S$.  Indeed, we have
  \begin{align*}
      (\sgn(\sigma)\sigma +\sgn(\sigma')\sigma') \beta(R)\mathbf{x}_T^S &= (\sgn(\sigma)\sigma +\sgn(\sigma)\sgn(\rho_\sigma)\sigma\rho_\sigma)\beta(R) \mathbf{x}_T^S \\
      &= (\sgn(\sigma)\sigma -\sgn(\sigma)\sigma\rho_\sigma)\beta(R)\mathbf{x}_T^S \\
      &= (\sgn(\sigma)\sigma \beta(R)- \sgn(\sigma)\sigma \rho_\sigma\beta(R))\mathbf{x}_T^S \\
      &= (\sgn(\sigma)\sigma \beta(R)- \sgn(\sigma)\sigma \beta(R))\mathbf{x}_T^S \\
      &=0
  \end{align*}
  where the last computation follows because $\rho_\sigma\in R$ and therefore $\rho_\sigma$ permutes the terms of $\beta(R)$.  It follows that $\alpha(U)F_T^S=0$, as desired.
\end{proof}

The above lemma is the exact analog to Lemma 3.3 in \cite{Peel}.  Using this lemma, the proof of Proposition \ref{prop:Garnir} now exactly follows that of Theorem 3.1 in \cite{Peel} for the ordinary Specht polynomials, since the remaining steps of Peel's proof only depend on the operators $\alpha(U)$ and not on the specific polynomials they are applied to.  We therefore omit the rest of the details and refer to \cite{Peel}.  

It now follows that the elements $F_T^S$, for $T$ a standard tableau of shape $\mathrm{shape}(S)$, span the space $V^S$.  Finally, we show the polynomials $F_T^S$ for $T\in \SYT(n)$ are linearly independent in $\mathbb{C}[x_1,\ldots,x_n]$.  In fact, their images are independent in the coinvariant ring $R_n$.

To prove this, we use the \defn{last letter order} $\lessdot$ on standard Young tableaux defined in \cite{ATY}.  In particular, for any two standard tableaux $T_1,T_2$ of the same shape, let $m(T_{1},T_2)$ be the largest letter that is not in the same square in $T_1$ as in $T_2$.  Then we say $T_1\lessdot T_2$ if $m(T_1,T_2)$ is in row $\ell$ in $T_1$ and row $k$ in $T_2$ with $\ell<k$.

\begin{example}
  The last letter order on the shape $(2,4)$ puts the standard tableaux in the following order from least to greatest:
  $$\young(24,1356), \hspace{0.3cm} \young(34,1256), \hspace{0.3cm}\young(25,1346), \hspace{0.3cm}\young(35,1246),\hspace{0.3cm}\young(45,1236),$$
  $$\young(26,1345), \hspace{0.3cm}\young(36,1245), \hspace{0.3cm}\young(46,1235), \hspace{0.3cm}\young(56,1234)$$

\end{example}

We also require the following elementary linear algebra fact, whose proof we omit.

\begin{lemma}
  Let $V$ and $W$ be vector spaces over a field $k$ of characteristic $0$, with a nondegenerate bilinear form $\langle,\rangle:V\times W\to k$.  Let $v=\{v_1,\ldots,v_r\}\subseteq V$ and $w=\{w_1,\ldots,w_r\}\subseteq W$, and suppose $\langle v_i,w_j\rangle=0$ whenever $i<j$ and further that $\langle v_i,w_i\rangle\neq 0$ for all $i$.  Then $v$ and $w$ are both independent sets of vectors in $V$ and $W$ respectively. 
\end{lemma}

\begin{proposition}
  For a fixed $S\in \SSYT(\lambda)$, the polynomials $F_T^S$, for $T\in \SSYT(\lambda,(1^n))$, are independent in the coinvariant ring $R_n$.
\end{proposition}

\begin{proof}
  We make use of the bilinear form and ordering on tableaux defined in \cite{ATY}.  In particular, for $f,g\in R_n$ define $$\langle f,g\rangle=\frac{1}{\Delta}\sum_{\sigma\in S_n} \sgn(\sigma)\sigma(\tilde{f}\tilde{g})\mid_{x_1=x_2=\cdots=x_n=0}$$  where $\tilde{f}$ and $\tilde{g}$ are lifts of $f$ and $g$ in $\mathbb{C}[x_1,\ldots,x_n]$ and $\Delta=\prod_{i<j}(x_i-x_j)$ is the Vandermonde determinant. 

Note that by the construction of the cocharge labels, there is always a permutation of the tuple $(0,1,2,\ldots,n-1)$ that is elementwise greater than the exponents of $x_T^S$.  Thus one can construct a \textit{complimentary monomial} $m_{T}^S$ to $x_T^S$ in the following way: order the factors $x_i^j$ in $x_T^S$ from largest to smallest $j$, breaking ties from largest to smallest $i$, and order the subscripts $i_1,\ldots,i_n$ in this order.  Then assign each $x_i$ the unique exponent in $m$ such that the product $m_T^S\cdot x_T^S=x_{i_1}^0x_{i_2}^1 x_{i_3}^2\cdots x_{i_n}^{n-1}$.
  
  Now, define $G_T^S=\varepsilon_{T'}\mathbf{m}_T^{S}$ where $T'$ is the transpose of the standard Young tableau $T$.  Then by an identical computation as in \cite{ATY} (in the last paragraph of the proof of Proposition 1 part (2)), we have that $$\langle F_{T_1}^S, G_{T_2}^S\rangle$$ is nonzero if $T_1=T_2$, since the only surviving term in the computation is the product $x_{i_1}^0x_{i_2}^1x_{i_3}^2\cdots x_{i_n}^{n-1}$, whose antisymmetrization is, up to a sign, the Vandermonde determinant $\Delta$ itself.  Moreover, as in the proof of Proposition 2 of \cite{ATY}, we see that $\langle F_{T_1}^S, G_{T_2}^S\rangle$ is equal to $0$ if $T_1>T_2$ in the last letter order, as in this case $\varepsilon_{T_1}\varepsilon_{T_2}=0$ as operators (see \cite{ATY}).
  
  Thus we have an upper triangular transition matrix between the $F$ and $G$ polynomials, and so the polynomials $F_T^S$ for $T\in \Tab(\lambda)$ are independent in $R_n$.  
\end{proof}

\begin{corollary}
  The space $V^S$ is a copy of the irreducible $S_n$-module $V^\lambda$.
\end{corollary}


\subsection{Independence in $R_\mu$ for two-row shapes}

We now show that, for two-row shapes $\mu$, the set of semistandard higher Specht polynomials for $R_\mu$ is independent in $R_\mu$, by induction on the size of $\mu$.  Our main tool is a recursion developed by Garsia and Procesi \cite{GarsiaProcesi}.  We recall their notation as follows.

\begin{definition}
  Let $\mu=(\mu_1,\ldots,\mu_r)$ be a partition and let $i\le r$.  Then $\mu^{(i)}$ is the partition whose parts  are $\mu_1,\ldots,\mu_{i-1},\mu_i-1,\mu_{i+1},\ldots,\mu_r$ (not necessarily in nonincreasing order).
\end{definition}  

\begin{example}
  If $\mu=(3,3,2)$, then $\mu^{(1)}=\mu^{(2)}=(3,2,2)$, and $\mu^{(3)}=(3,3,1)$. 
\end{example} 

Garsia and Procesi \cite{GarsiaProcesi} show that $$R_\mu=\bigoplus_{i=1}^{\mu'_1} x_n^{i-1}R_{\mu}/x_n^{i}R_\mu$$ as vector spaces.  Moreover, considered as $S_{n-1}$-modules, we have $$R_{\mu^{(i)}}\cong x_n^{i-1}R_{\mu}/x_n^{i}R_\mu$$  via the map $p\mapsto x_n^{i-1}p$.  It follows that there is an $S_{n-1}$-module decomposition $$R_\mu=\bigoplus_{i=1}^{\mu_1'} R_{\mu^{(i)}}.$$  We therefore can conclude the following.

\begin{lemma}\label{lem:induct}
  Let $\mu\vdash n$ and suppose $\mathcal{C}(\mu^{(i)})$ is a basis of $R_{\mu^{(i)}}$ for each $i=1,2,\ldots,\mu'_1$.  Then $\bigcup x_n^{i-1} \mathcal{C}(\mu^{(i)})$ is a basis of $R_\mu$.
\end{lemma}

With this in mind, we outline the following general strategy for proving that $\B_\mu$ is a basis for $R_\mu$.  We assume for induction that $\B_{\lambda}$ is a basis for $R_\lambda$ for all smaller shapes $\lambda$ contained in  $\mu$.  Then, we define $\mathcal{C}_\mu=\bigcup_{i}x_n^{i-1}\B_{\mu^{(i)}}$, which is a basis of $R_\mu$ by the induction hypothesis and Lemma \ref{lem:induct}.  Finally, if we can show that the transition matrix between $\mathcal{B}_\mu$ and $\mathcal{C}_\mu$ is invertible, then the induction is complete. 

We can further simplify this process by noting that we can restrict to basis elements of a given degree.

\begin{definition}
  For any set of homogeneous polynomials $\B$, we write $\B^{(d)}$ to denote the subset of degree $d$ polynomials in $\B$. 
\end{definition}

Note that it suffices to show that the transition matrix between $\mathcal{B}_\mu^{(d)}$ and $\C_\mu^{(d)}$ is invertible for every $d$, since both sets consist of homogeneous polynomials and $R_\mu$ is degree graded.  We will implement this inductive approach for two-row shapes below, by showing that the transition matrix is in fact lower triangular in this case.  We begin by illustrating this phenomenon with an example.

\begin{example}
Consider the case $\mu=(3,3)$ and $d=2$.  Figure \ref{fig:matrix1} shows the transition matrix from $\B_\mu^{(2)}$ to $\C_\mu^{(2)}$.  Here, the elements of $\B_\mu^{(2)}$ are written in last letter order down the left hand side of the table.  The elements of $\C_\mu^{(2)}=\B_{(3,2)}^{(2)}\cup x_6 \B_{(3,2)}^{(1)}$ are written across the top, with the elements from $\B_{(3,2)}^{(2)}$ coming before those of $x_6 \B_{(3,2)}^{(1)}$, with ties broken in last letter order.  If a coefficient is $0$ we leave that entry blank.
 \begin{figure}
     \centering
       \renewcommand{\arraystretch}{1.3}
  $\begin{array}{lccccccccc}
       & F_{\tiny \young(24,135)}^{S'} & F_{\tiny \young(34,125)}^{S'}
       & F_{\tiny \young(25,134)}^{S'}
       & F_{\tiny \young(35,124)}^{S'}
       & F_{\tiny \young(45,123)}^{S'}
       & x_6F_{\tiny \young(2,1345)}^{S''}
       & x_6F_{\tiny \young(3,1245)}^{S''}
       & x_6F_{\tiny \young(4,1235)}^{S''}
       & x_6F_{\tiny \young(5,1234)}^{S''} \\
    F_{\tiny\young(24,1356)}^{S} & 4 &  &  &  &  &  &  &  &  \\
    F_{\tiny\young(34,1256)}^{S} &  & 4 &  &  &  &  &  &  &  \\
    F_{\tiny\young(25,1346)}^{S} &  &  & 4 &  &  &  &  &  &  \\
    F_{\tiny\young(35,1246)}^{S} &  &  &  & 4 &  &  &  &  &  \\
    F_{\tiny\young(45,1236)}^{S} &  &  &  &  & 4 &  &  &  &  \\
    F_{\tiny\young(26,1345)}^{S} & 4/3 &  & 4/3 &  &  & 8/3 &  &  &  \\
    F_{\tiny\young(36,1245)}^{S} &  & 4/3 &  & 4/3 &  &  & 8/3 &  &  \\
    F_{\tiny\young(46,1235)}^{S} & -4/3 & 4/3 &  &  & 4/3 &  &  & 8/3 &  \\
    F_{\tiny\young(56,1235)}^{S} & 4/3 &  & -8/3 & 4/3 & 4/3 &  &  &  & 8/3 \\
  \end{array}$
     \caption{The transition matrix that expresses the elements of $\B_{(3,3)}^{(2)}$ (the row labels) in terms of those of $\C_{(3,3)}^{(2)}=\B_{(3,2)}^{(2)}\cup x_6 \B_{(3,2)}^{(1)}$ (the column labels).}
     \label{fig:matrix1}
 \end{figure}

  Here, $$S={\tiny\young(22,1112)},\hspace{0.5cm} S'={\tiny \young(22,111)},\hspace{0.5cm}S''={\tiny\young(2,1112)}.$$
  
  We will show in the proof of Theorem \ref{thm:two-rows} that, if the largest number $n$ is in the bottom row of $T$, then $F_T^S=cF_{T'}^{S'}$ for some constant $c$, where $T'$ is formed by removing the largest entry $n$ from $T$.  On the other hand, if the largest number $n$ is in the top row of $T$, then we will show that $$
    F_{T}^S=\alpha x_n F_{T'}^{S''} +\beta \sum_{j=b_{d+1}}^{b_{n-d}} F_{T'_j}^{S'} $$
   where $\alpha=\frac{d}{n-2d+1}+d$ and $\beta=\frac{n-d}{n-2d+1}$, and where $T'_j$ is the tableau formed by removing $j$ from the bottom row of $T'$ and inserting it in the top row. Here $n=6$ and $d=2$, so $\alpha=8/3$ and $\beta=4/3$.  Thus, for instance, $$F_{\tiny\young(26,1345)}^{S}=\frac{8}{3} x_6 F_{\tiny\young(2,1345)}^{S''}+\frac{4}{3}F_{\tiny\young(24,135)}^{S'}+\frac{4}{3}F_{\tiny\young(25,134)}^{S'}.$$  Indeed, subtracting the right hand side from the left hand side of the above equation yields the polynomial $$-\frac{8}{3}(x_2-x_1)(x_3+x_4+x_5+x_6)=-\frac{8}{3}(e_2(x_2,x_3,x_4,x_5,x_6)-e_2(x_1,x_3,x_4,x_5,x_6))\in I_\mu.$$
 We similarly have $$F_{\tiny\young(46,1345)}^{S}=\frac{8}{3} x_6 F_{\tiny\young(4,1235)}^{S''}+\frac{4}{3}F_{\tiny\young(43,125)}^{S'}+\frac{4}{3}F_{\tiny\young(45,123)}^{S'}.$$  The second summand is not a basis element, but we can straighten it using the Garnir relations to express it in terms of $F_{T'}^{S'}$ elements where $T'$ is standard, to obtain the second to last row of the matrix above.
\end{example}

The following lemma will be used repeatedly in the proof of Theorem \ref{thm:two-rows}.
 
 \begin{lemma}
   Suppose $\mu$ is a two-row shape and $S$ has content $\mu$.  Then $S$ has at most two rows, say of lengths $d$ and $n-d$.  If $T$ is the tableau of the same shape as $S$ with entries $t_1,\ldots,t_d$ in the top row and $b_1,\ldots,b_{n-d}$ in the bottom, we have \begin{equation}\label{eqn:two-row}
     F_T^S=d!(n-d)!\prod_{i=1}^{d} (x_{t_i}-x_{b_i}).
 \end{equation}
 \end{lemma}
 
 \begin{proof}
 The tableau $S$ looks like: $$\young(222,111111222)$$ where there are $\mu_1$ $1$'s, $\mu_2$ $2$'s, and exactly $d$ of the $2$'s are in the top row.  Thus the cocharge indices are all $0$ in the first row and $1$ in the second.  The equation follows.
 \end{proof}
 
   We now prove Theorem \ref{thm:two-rows}, which we restate here for the reader's convenience.

{
\renewcommand{\thetheorem}{\ref{thm:two-rows}}
\begin{theorem}
If $\mu=(n-k,k)$ for some $k\ge 0$, the set $$\BBB_\mu=\{F_T^S:(S,T)\in \bigcup_{\lambda\vdash n}\SSYT(\lambda,\mu)\times \SYT(\lambda)\}$$ descends to a higher Specht basis of $R_\mu$.  In other words, Conjecture \ref{conj:R-mu} holds for one- and two-row shapes.
\end{theorem}
\addtocounter{theorem}{-1}
}

\begin{proof}
The base case, $n=1$, holds trivially for the unique partition $\mu=(1)$. 

Let $\mu=(n-k,k)$ and assume for induction that the claim holds for all smaller two-row (or one-row) shapes fitting inside $\mu$.  In particular, it holds for $\mu^{(1)}$ and $\mu^{(2)}$.   Then by Lemma \ref{lem:induct}, the set $$\mathcal{C}_\mu:=\mathcal{B}_{\mu^{(1)}}\cup x_n \mathcal{B}_{\mu^{(2)}}$$ is a basis for $R_\mu$.

Let $t=|\B_\mu|=|\C_\mu|=\binom{n}{\mu}$.  We will show there are total orderings $b_1,\ldots,b_t$ and $c_1,\ldots,c_t$ on $\B_\mu$ and $\C_\mu$ respectively for which $$b_i=\sum_{j\le i}\alpha_{i,j} c_j$$ for some constants $\alpha_{i,j}$ with $\alpha_{i,i}\neq 0$.  Since the transition matrix $[\alpha_{i,j}]$ is lower triangular with a nonzero diagonal, it will follow that $\B_\mu$ is a basis of $R_\mu$.

To define these orderings, first note that the sets $\B_\mu$ and $\C_\mu$ both consist of homogeneous polynomials, and $R_\mu$ is graded by degree.  We therefore can define $b_i<b_j$ if $\deg(b_i)<\deg(b_j)$ and similarly $c_i<c_j$ if $\deg(c_i)<\deg(c_j)$.  With respect to this partial ordering, we have $\alpha_{i,j}=0$ if $i<j$.  Thus it suffices to choose a fixed degree $d$ and consider just the basis elements $b_i$ and $c_i$ of degree $d$, and choose an appropriate total ordering on the corresponding subsets $\B^{(d)}_\mu$ and $\C^{(d)}_\mu$ to show that the corresponding sub-matrix $M^{(d)}$ is lower triangular.

Since the cocharge of a tableau with only $1$'s and $2$'s is equal to the size of the top row, the elements in $\B^{(d)}_\mu$ are precisely those of the form $F_{T}^S$ where $S$ is the unique tableau of shape $\lambda=(n-d,d)$ and content $\mu$, and $T\in \SYT(\lambda)$.   Since $S$ is fixed, we define our ordering based on $T$.  In particular we define $F_{T_1}^S<F_{T_2}^{S}$ if and only if $T_1\lessdot T_2$ in the last letter order.  (See the row ordering of the matrix $M^{(d)}$ in Figure \ref{fig:matrix1} for an example.)

To order the elements of $\C^{(d)}_\mu$, let $S'$ be the unique tableau of content $\mu^{(1)}$ and shape $\lambda^{(1)}=(n-d-1,d)$ (where if $n-d=d$ then $S'$ is undefined), and let $S''$ be the unique tableau of content $\mu^{(2)}$ and shape $\lambda^{(2)}=(n-d,d-1)$.  Then we have 
\begin{align*}
\C^{(d)}_\mu&=\B_{\mu^{(1)}}^{(d)}\cup x_n \B_{\mu^{(2)}}^{(d-1)} \\
&= \{F_{T'}^{S'}\}\cup \{x_nF_{T''}^{S''}\}
\end{align*} where in the first set above $T'\in \SYT(\lambda^{(1)})$ and in the second, $T''\in \SYT(\lambda^{(2)})$.  We enforce that the elements $F_{T'}^{S'}$ come before those of the form $x_nF_{T''}^{S''}$ in our ordering, and then we break ties by the last letter order on the subscripts $T'$ and $T''$ respectively.  (See the column ordering of the matrix $M^{(d)}$ in Figure \ref{fig:matrix1}.)

Now, consider the set $\B_0$ of elements $F_T^S\in \B_\mu^{(d)}$ for which $n$ is in the bottom row of $T$ (so necessarily $d\neq n-d$).  Note that $\B_0$ forms an initial sequence of the total ordering on $\B_{\mu}^{(d)}$.  Removing $n$ from the bottom row of such a tableau $T$ forms a standard tableau $T'$ of shape $\lambda^{(1)}=(n-d-1,d)$.  We claim that $$F_T^S=cF_{T'}^{S'}$$ for some constant $c$.  Indeed, let $t_1,\ldots,t_d$ be the entries in the top row of $T$, and let $b_1,\ldots,b_d$ be the first $d$ entries in the bottom row; then by equation (\ref{eqn:two-row}), both polynomials are nonzero constant multiples of $$(x_{t_1}-x_{b_1})(x_{t_2}-x_{b_2})\cdots (x_{t_k}-x_{b_k}).$$  Thus the sets $\B_0$ and $\B_{\mu^{(1)}}^{(d)}$, which are both initial sequences of their respective orderings, are scalar multiples of one another, and so the transition matrix $M^{(d)}$ is block lower triangular, of the form $$\left(\begin{array}{cc}
    c I & 0 \\
    X & Y
\end{array}\right).$$  It remains to show that $Y$ is lower triangular with nonzero diagonal entries.  We in fact will show that $Y=\alpha I$ for some constant $\alpha$ as well.

Indeed, let $\B_1$ be the set of elements of $\B_\mu^{(d)}$ of the form $F_T^S$ where $n$ is in the top row of $T$.  Let $T$ be such a tableau, with top row having entries $t_1,\ldots,t_d=n$ and bottom row having entries $b_1,\ldots,b_{n-d}$. Define $T''$ to be the tableau formed by deleting $n$ from $T$, and define the tableau $T'_{j}$ for $j\in \{b_{d+1},\ldots,b_{n-d}\}$ to be the tableau formed by deleting $j$ from the bottom row of $T''$ and placing it at the end of the top row.  Note that $T'_j$ may not be standard.  However, since the Garnir relations are satisfied, $F_{T'_j}^{S'}$ is a linear combination of the polynomials $F_{T'}^{S'}$ where $T''$ is standard, which come before elements of the form $x_n F_{T''}^{S''}$ in the ordering on $\C_\mu^{(d)}$. 

We will show that \begin{equation}\label{eqn:two-row2}
    F_{T}^S=\alpha x_n F_{T''}^{S''} +\beta \sum_{j=b_{d+1}}^{b_{n-d}} F_{T'_j}^{S'} 
\end{equation} for some nonzero constants $\alpha$ and $\beta$.  In light of the Garnir relations and the ordering, it will follow that $Y=\alpha I$ as claimed.

To show (\ref{eqn:two-row2}), set $$\alpha=\frac{d}{n-2d+1}+d$$ and $$\beta=\frac{n-d}{n-2d+1}.$$  Then we have, using (\ref{eqn:two-row}) repeatedly:
\begin{align*}
    F_T^S-\alpha x_n F_{T''}^{S''} -\beta \sum_{j=b_{d+1}}^{b_{n-d}} F_{T'_j}^{S'} &= d!(n-d)!\prod_{i=1}^d(x_{t_i}-x_{b_i}) \\ &\phantom{=} -\alpha x_n (d-1)!(n-d)!\prod_{i=1}^{d-1}(x_{t_i}-x_{b_i})
    \\
    &\phantom{=}-\beta \sum_{j=b_{d+1}}^{b_{n-d}}d!(n-d-1)! (x_j-x_{b_{d}})\prod_{i=1}^{d-1}(x_{t_i}-x_{b_i}).
\end{align*}

We wish to show that the right hand side is equal to $0$ in $R_\mu$.  Thus we may divide the right hand side by $(d-1)!(n-d-1)!$, and as a shorthand define $P=\prod_{i=1}^{d-1}(x_{t_i}-x_{b_i})$, so that we wish to show that the simpler expression $$P\cdot\left(d(n-d)(x_{n}-x_{b_d})-\alpha (n-d) x_n-d\beta\sum_{j=b_{d+1}}^{b_{n-d}} (x_{j}-x_{b_d})\right)$$ is $0$ in $R_\mu$, that is, it lies in the ideal $I_\mu$.
In the parenthetical above, substituting $\alpha$ and $\beta$ in for the expressions, it is easily verified that the coefficients of $x_n$, $x_{b_d}$, and each $x_j$ for $j=b_{d+1},\ldots,b_{n-d}$ are all equal to $-d(n-d)/(n-2d+1)$.  Thus the entire expression is a constant multiple of \begin{equation}\label{eqn:simplified-two-row}P\cdot (x_{b_{d}}+x_{b_{d+1}}\cdots+x_{b_{n-d}}+x_n).\end{equation} 

Finally, we show that this expression is in $I_\mu$.  Note that $e_d(X)\in I_\mu$ for any set $X$ of $n-d+1$ variables by the definition of the Tanisaki generators (Equation (\ref{eq:inequality})) and the fact that $\mu$ has two rows, the second of which is at least $d$.  Thus $$e_{d}(\{x_{r_1},\ldots,x_{r_{d-1}}\}\cup \{x_{b_d},x_{b_{d+1}},\ldots,x_{b_{n-d}},x_n\})\in I_\mu$$ for any choice of subscripts in which $r_i$ is either equal to $t_i$ or $b_i$ for each $i=1,\ldots,d-1$.  We assign this partial elementary symmetric function a sign of $-1$ if there are an odd number of $r_i$ subscripts equal to $b_i$, and a sign of $1$ otherwise.  Summing these signed functions yields the expression (\ref{eqn:simplified-two-row}).
\end{proof}

\subsection{Beyond two-row shapes}

In this section, we provide computer evidence that our inductive approach above may be able to be extended to all partition shapes. 

First, Conjecture \ref{conj:R-mu} has been verified using Sage \cite{Sage} for all partition shapes $\mu$ of size at most $7$.  We have also verified it for the three-row shape $(3,3,2)$ of size $8$, which is often the smallest shape in which conjectures related to cocharge start to break down (see, for instance, \cite{Gillespie}, in which a property of cocharge is proven combinatorially for all shapes of the form $(a,b,1^k)$, but the method does not extend to any other three row shapes).

Second, while the transition matrix expressing $\mathcal{B}_\mu^{(d)}$ in terms of $\mathcal{C}_\mu^{(d)}$ is not always lower-triangular for partition shapes $\mu$ having more than two rows, it is very nearly so, in the following sense. 

\begin{definition}
  We say an $n\times n$ matrix $M$ is \defn{almost lower triangular} if there is an upper triangular $n\times n$ matrix $A$ for which $MA$ is lower triangular with nonzero diagonal entries.
\end{definition}

Clearly every invertible lower triangular matrix is almost lower triangular, and every almost lower triangular matrix is invertible.  Computer evidence indicates that there always exist orderings on the sets $\mathcal{B}_\mu^{(d)}$ and $\mathcal{C}_\mu^{(d)}$ such that the transition matrix between them in $R_\mu$ is almost lower triangular.

For example, the transition matrix for $\mu=(3,1,1)$ and $d=2$ is:
$$M=\left(\begin{array}{ccccccccc}
        1 & 0 & 0 & 0 & 0 & 1 & 0 & 0 & 0  \\
        0 & 1 & 0 & 0 & 0 & 0 & 1 & 0 & 0 \\
        0 & 0 & 1 & 0 & 0 & 0 & 0 & 1 & 0 \\
        0 & 0 & 0 & 1 & 0 & 0 & 0 & 0 & 0 \\
        0 & 0 & 0 & 0 & 1 & 0 & 0 & 0 & 0 \\
     -1/2 & 0 & 0 &1/2& 0 & 1 & 0 & 0 & 0 \\
       0 &-1/2& 0 & 0 &1/2& 0 & 1 & 0 & 0 \\
       0 & 0 &-1/2&-1/2&1/2& 0 & 0 & 1 & 0 \\
        1/4 & 1/4 & 1/4 & 0 & 0 & 1/4 & 1/4 & 1/4 & -5/4 
    \end{array}\right)
$$
which is almost lower triangular.  Indeed, multiplying $M$ on the right by the upper triangular matrix $$A=
    \left(\begin{array}{ccccccccc}
        1 & 0 & 0 & 0 & 0 & -1 & 0 & 0 & 0  \\
        0 & 1 & 0 & 0 & 0 & 0 & -1 & 0 & 0 \\
        0 & 0 & 1 & 0 & 0 & 0 & 0 & -1 & 0 \\
        0 & 0 & 0 & 1 & 0 & 0 & 0 & 0 & 0 \\
        0 & 0 & 0 & 0 & 1 & 0 & 0 & 0 & 0 \\
        0 & 0 & 0 & 0 & 0 & 1 & 0 & 0 & 0 \\
        0 & 0 & 0 & 0 & 0 & 0 & 1 & 0 & 0 \\
        0 & 0 & 0 & 0 & 0 & 0 & 0 & 1 & 0 \\
        0 & 0 & 0 & 0 & 0 & 0 & 0 & 0 & 1
    \end{array}\right)$$
yields a lower triangular matrix with nonzero diagonal entries.

\section{A higher Specht basis for $R_{n,k,(n-1)}$}\label{sec:Rnkmu}

We now combine the methods of the previous two sections to prove Theorem \ref{thm:Rnkmu}, which we restate here for the reader's convenience.

{
\renewcommand{\thetheorem}{\ref{thm:Rnkmu}}
\begin{theorem}
 Consider the set of polynomials $$\mathcal{B}_{n,k,(n-1)}=\{F_T^S\cdot e_1^i\}$$ where $F_T^S\in \mathcal{B}_{(n-1,1)}$ is a semistandard higher Specht polynomial for the shape $(n-1,1)$, and $i<k-\des(S)$.  Then $\mathcal{B}_{n,k,(n-1)}$ descends to a higher Specht basis for $R_{n,k,(n-1)}$. 
\end{theorem}
\addtocounter{theorem}{-1}
}
\begin{proof}
Since $S_n$ acts trivially on the elementary symmetric function $e_1$, if $\mathcal{B}_{n,k,(n-1)}$ is a basis then it is indeed a higher Specht basis.  In particular, the polynomials $F_T^S\cdot e_1^i$ for a fixed $i$ and for a fixed tableau $S$ of shape $\lambda$ span a copy of the irreducible representation $V^\lambda$ of $S_n$.

To show that $\mathcal{B}_{n,k,(n-1)}$ is a basis,  we make use of a short exact sequence for the modules $R_{n,k,\mu}$ that is analogous to the sequence (\ref{eq:SES}) for $R_{n,k,s}$ used in Section \ref{sec:Rnk}.  Griffin shows \cite[Lem. 4.12]{Griffin} that there is a short exact sequence of $S_n$-modules $$0\to R_{n,k,\mu}\to R_{n,k+1,\mu} \to R_{n,k+1,\mu+(1)}\to 0$$ for any $k<n$ and $\mu$ for which $R_{n,k,\mu}$ is defined.  Here the notation $\mu+(1)$ indicates that we simply add one part of size $1$ to the partition $\mu$.  In the sequence above, the first nontrivial map is multiplication by $e_{n-|\mu|}$ and the second is given by setting $e_{n-|\mu|}=0$.

Setting $\mu=(n-1)$, we have the exact sequences $$0\to R_{n,k,(n-1)}\to R_{n,k+1,(n-1)}\to R_{n,k+1,(n-1,1)}\to 0$$ for any $k\ge 1$.

We now prove the claim by induction on $k$.  For the base case $k=1$, note that we have $R_{n,1,(n-1)}=\mathbb{C}[x_1,\ldots,x_n]/I_{n,1,(n-1)}$, where the ideal $I_{n,1,(n-1)}$ includes all the variables $x_1,\ldots,x_n$ as generators, since $k=1$.  Hence we simply have $R_{n,1,(n-1)}\cong \mathbb{C}$, generated by the single basis element $1$.  The set $\mathcal{B}_{n,1,(n-1)}$ consists of all polynomials $F_T^S\cdot e_1^i$ for which $S$ has content $(n-1,1)$ and $i<1-\des(S)$, which forces $\des(S)=0$ and $i=0$.  The only such tableau $S$ is $$S=\young(1111\cdots  112)$$ which forces $$T=\young(123\cdots n),$$ and these give rise to the unique basis element $F_T^S=1$.

For the induction step, let $k\ge 2$ and assume the claim holds for all smaller $k$.  Note that since $(n-1,1)$ is a partition of $n$, the right-hand module $R_{n,k+1,(n-1,1)}$ of the exact sequence is simply the Garsia-Procesi module $R_{(n-1,1)}$ for any $k\ge 1$.  Hence, by Theorem \ref{thm:two-rows}, a higher Specht basis for this module is given by $\mathcal{B}_{(n-1,1)}$.   

By the induction hypothesis, the left hand term of the exact sequence has $\mathcal{B}_{n,k,(n-1)}$ as a basis.   It follows that the middle term $R_{n,2,(n-1)}$ has basis $e_1\mathcal{B}_{n,k,(n-1)}\cup \mathcal{B}_{(n-1,1)}$.  By the definition of the bases, this is simply equal to $\mathcal{B}_{(n,k+1,(n-1))}$, and the proof is complete.
\end{proof}

\end{document}